%% file: subgradient_method.tex
\documentclass[12pt]{article}
\usepackage[margin =1in]{geometry}
\usepackage{amssymb,amsthm}
\usepackage{amsmath}
\usepackage[bookmarks=false]{hyperref}
\usepackage{cite}
\usepackage{color}

\numberwithin{equation}{section}

\newcommand{\inclu}[0] {\ar@{^{(}->}}

\newcommand{\dist}{{\rm dist}}
\newcommand{\R}{\mathbb{R}}

\newcommand{\EE}{\mathbb{E}}

\newcommand{\RR}{\mathbb{R}}

\newcommand{\cF}{\mathcal{F}}

\newcommand{\proj}{\mathrm{proj}}
\newcommand{\cX}{\mathcal{X}}

\newcommand{\prox}{{\rm prox}}
\newcommand{\dom}{\text{dom}\,}
\newcommand{\argmin}{\operatornamewithlimits{argmin}}

\newtheorem{thm}{Theorem}[section]

\newtheorem{lem}[thm]{Lemma}
\newtheorem{cor}[thm]{Corollary}

\theoremstyle{remark}

\usepackage{mathtools}
\DeclarePairedDelimiter{\dotp}{\langle}{\rangle}
\usepackage[boxruled]{algorithm2e}

\begin{document}
	
	\title{Stochastic subgradient method converges at the rate $O(k^{-1/4})$ on weakly convex functions}
	
	
	\author{Damek Davis\thanks{School of Operations Research and Information Engineering, Cornell University,
			Ithaca, NY 14850, USA;
			\texttt{people.orie.cornell.edu/dsd95/}.}
		\and 
		Dmitriy Drusvyatskiy\thanks{Department of Mathematics, U. Washington, 
			Seattle, WA 98195; \texttt{www.math.washington.edu/{\raise.17ex\hbox{$\scriptstyle\sim$}}ddrusv}. Research of Drusvyatskiy was supported by the AFOSR YIP award FA9550-15-1-0237 and by the NSF DMS   1651851 and CCF 1740551 awards.}
	}

	\date{}
	\maketitle
	\begin{abstract}
	We prove that the proximal stochastic subgradient method, applied to a weakly convex problem, drives the gradient of the Moreau envelope to zero at the rate $O(k^{-1/4})$. As a consequence, we resolve an open question on the convergence rate of the proximal stochastic gradient method for minimizing the sum of a smooth nonconvex function and a convex proximable function.
	\end{abstract}

	\section{Introduction}
	In this work, we consider the optimization problem 
	\begin{equation}\label{eqn:problem_class}
	\min~ \varphi(x) := g(x) + r(x)
	\end{equation}
	under the following  assumptions on the functional components $g$ and $r$.
	Throughout, $r \colon \RR^d \rightarrow \RR \cup \{\infty\}$ is a closed convex function with a computable proximal map
	\begin{equation*}
	\prox_{\alpha r}(x):=\argmin_{y}\, \left\{r(y)+\tfrac{1}{2\alpha}\|y-x\|^2\right\},
	\end{equation*}
	 while $g\colon\R^d\to\R$  is a
	$\rho$-weakly convex function, meaning that the assignment $x\mapsto g(x)+\frac{\rho}{2}\|x\|^2$ is convex. The above assumptions on $r$ are standard in the literature (see e.g. \cite{nest_conv_comp,beck,Parikh:prox}), while those on $g$ deserve some commentary.
	The class of weakly convex functions, first introduced in English in~\cite{Nurminskii1973}, is  broad. Indeed,  it includes all convex functions and smooth functions with Lipschitz continuous gradient. More generally,  any function of the form $g = h\circ c$, with $h$ convex and Lipschitz and $c$ a smooth map with Lipschitz Jacobian~\cite[Lemma 4.2]{comp_DP}, is weakly convex.  Classical literature highlights the importance of weak convexity in optimization \cite{fav_C2,amen,prox_reg}, while recent advances in statistical learning and signal processing have further reinvigorated the problem class \eqref{eqn:problem_class}. For a recent discussion on the role of weak convexity in large-scale optimization, see for example \cite{prox_point_surv}. 
	
	The proximal subgradient method is perhaps the simplest algorithm for the problem   \eqref{eqn:problem_class}. Given a current iterate $x_t$, the method repeats the steps
\begin{equation*}\left\{
\begin{aligned}
&\textrm{Choose } \zeta_t\in \partial g(x_t)\\
	& \textrm{Set } x_{t+1}= \prox_{\alpha_t r} (x_t- \alpha_t\zeta_t)
	\end{aligned}\right\},
	\end{equation*}
where $\alpha_t>0$ is an appropriately chosen control sequence. Here, the subdifferential $\partial g$ is meant in a standard  variational analytic sense \cite[Definition 8.3]{RW98}; we will recall the precise definition in Section~\ref{sec:main}. The setting when $r$ is the indicator function of a closed convex set $\mathcal{X}$ reduces the algorithm to the  classical {projected  subgradient method}. Indeed, then the proximal map $\prox_{\alpha_t r}(\cdot)$ is simply the nearest point projection $\proj_{\cX}(\cdot)$.

The primary goal in nonsmooth nonconvex optimization is the search for stationary points. 
A point $x\in\R^d$ is called {\em stationary} for the problem \eqref{eqn:problem_class} if the inclusion $0\in \partial \varphi(x)$ holds. In ``primal terms'', these are precisely the points where the directional derivative of $\varphi$ is nonnegative in every direction \cite[Proposition 8.32]{RW98}: 
\begin{equation}\label{eqn:subdif_direc_der}
\dist(0;\partial \varphi(x))=-\inf_{v:\, \|v\|\leq 1} \varphi'(x;v).
\end{equation}

It has been known since \cite{Nurminskii1974,Nurminskii1973} that the (stochastic) subgradient method with $r=0$ generates an iterate sequence that subsequentially converges to a stationary point of the problem. A long standing open question in this line of work is to determine the ``rate of convergence'' of the basic (stochastic) subgradient method and of its proximal extensions. 

An immediate difficulty in addressing this question is that it is not a priori clear how to measure the progress of the algorithm. Indeed, neither the functional suboptimality gap, $\varphi(x_t)-\min \varphi$, nor the stationarity measure, $\dist(0;\partial \varphi(x_t))$, necessarily tend to zero along the iterate sequence. Instead, recent literature \cite{prox_point_surv,comp_DP} has identified a different measure of complexity of minimizing weakly convex functions, based on smooth approximations. 
The key construction we use is the {\em Moreau envelope}:
$$\varphi_{\lambda}(x):=\min_{y}~ \left\{\varphi(y)+\tfrac{1}{2\lambda}\|y-x\|^2\right\},$$
 where $\lambda > 0$. Standard results  show that as long as $\lambda<\rho^{-1}$, the envelope $\varphi_{\lambda}$ is $C^1$-smooth with the gradient  given by 
 \begin{equation}\label{eqn:grad_form}
\nabla \varphi_{\lambda}(x)=\lambda^{-1}(x-\prox_{\lambda \varphi}(x)).
\end{equation}
See for example \cite[Theorem 31.5]{rock}.
Moreover, the norm of the gradient $\|\nabla \varphi_{\lambda}(x)\|$ has an intuitive interpretation in terms of near-stationarity for the target problem \eqref{eqn:problem_class}. Namely, the definition of the Moreau envelope directly implies that for any $x\in\R^d$, the proximal point $\hat x:=\prox_{\lambda \varphi}(x)$ satisfies
\begin{equation*}
	\left\{\begin{array}{cl}
		\|\hat{x}-x\|&=  \lambda\|\nabla \varphi_{\lambda}(x)\|,\\ 
		\varphi(\hat x) &\leq \varphi(x),\\
		\dist(0;\partial \varphi(\hat{x}))&\leq \|\nabla \varphi_{\lambda}(x)\|.
	\end{array}\right. 
\end{equation*}
Thus a small gradient $\|\nabla \varphi_{\lambda}(x)\|$ implies that $x$ is {\em near} some point $\hat x$ that is {\em nearly stationary} for \eqref{eqn:problem_class}.
For a longer discussion of near-stationarity, see   \cite{prox_point_surv} or \cite[Section 4.1]{comp_DP}.

In this paper, we show that under an appropriate choice of the control sequence $\alpha_t$, the subgradient method will generate a point $x$ satisfying $\|\nabla \varphi_{1/2\rho}(x)\|\leq \varepsilon$ after at most $O(\varepsilon^{-4})$ iterations. A similar guarantee was recently established for the proximally guided projected subgradient method \cite{prixm_guide_subgrad}. This scheme proceeds by directly applying the gradient descent method to the Moreau envelope $\varphi_{\lambda}$, with each proximal point $\prox_{\lambda \varphi}(x)$ approximately evaluated by a convex subgradient method. In contrast, we show here that the basic subgradient method, without any modification or parameter tuning, already satisfies the desired convergence guarantees. This is perhaps surprising,  since neither the Moreau envelope $\varphi_{\lambda}(\cdot)$ nor the proximal map $\prox_{\lambda \varphi}(\cdot)$ explicitly appear in the definition of the subgradient method.

Though our results appear to be new even in this rudimentary deterministic set up, the argument we present applies  much more  
 broadly to stochastic proximal subgradient methods, in which only stochastic estimates of $\zeta_t$ are available. This is the setting of the paper. In this regard, we improve in two fundamental ways on the results in the seminal  papers \cite{ghad,Ghadimi2016mini,wotao}: first, we allow $g$ to be nonsmooth and second, we do not require the variance of our stochastic estimator for $\zeta_t$ to decrease as a function of $t$. The second contribution removes the well-known ``mini-batching" requirements common to~\cite{Ghadimi2016mini,wotao}, while the first significantly expands the class of functions for which the rate of convergence of the stochastic proximal subgradient method is known. It is worthwhile to mention that our techniques crucially rely on convexity of $r$, while~\cite{wotao} makes no such assumption. 

There is an extensive literature on stochastic subgradient methods in convex optimization, which we will not detail here; instead, we refer the interested reader to the seminal works~\cite{doi:10.1137/070704277,complexity}. An in-depth summary of recent work for nonconvex problems appears in~\cite{prixm_guide_subgrad}.

The outline of the paper is as follows. In the Section~\ref{sec:project_subgrad}, we present a simplified argument for the case in which $r$ is the indicator function of a closed convex set and the stochastic estimator has finite second moment. In this section, we also comment on improved rates in the convex setting. In Section~\ref{sec:proxmeth}, we prove convergence of the stochastic proximal subgradient method in full generality. In Section~\ref{sec:prox_gradient}, we modify the results of the previous section to the case in which $g$ is smooth and the stochastic estimator has finite variance.

\input{convergence_guarantees}

\let\oldbibliography\thebibliography
\renewcommand{\thebibliography}[1]{%
  \oldbibliography{#1}%
  \setlength{\itemsep}{-1pt}%
}

			\bibliographystyle{plain}
	\bibliography{bibliography}

\newpage	

\end{document}

%% file: convergence_guarantees.tex
\section{Convergence guarantees}\label{sec:main}
	Henceforth, we assume that the only access to $g$ is through a stochastic subgradient oracle.  Formally, we fix a probability space $(\Omega, \cF, P)$ and equip $\RR^d$ with the Borel $\sigma$-algebra.	
We make the following three standard assumptions: 
\begin{enumerate}
	\item[(A1)]\label{it1} It is possible to generate i.i.d.\ realizations $\xi_1,\xi_2, \ldots \sim dP$.
	\item[(A2)]\label{it2} There is an open set $U$ containing $\dom r$ and a measurable mapping $G \colon U \times \Omega \rightarrow \RR^d$ satisfying  $\EE_{\xi}[G(x,\xi)]\in \partial g(x)$ for all $x\in U$.
	\item[(A3)]\label{it3} There is a real $L \geq 0$ such that the inequality, $\EE_\xi\left[ \|G(x, \xi)\|^2\right] \leq L^2$, holds for all $x \in \dom r$. 
\end{enumerate}
Some comments are in order. First, the symbol ${\partial} g(x)$ refers to the {\em subdifferential}  of $g$ at $ x$. By definition, this is the set consisting of all vectors $v\in\R^d$ satisfying
\begin{equation*} \label{eqn:weak_subdiff}
g(y)\geq g( x)+\langle v,y- x\rangle +o(\|y- x\|)\qquad \textrm{ as }y\to  x.	
\end{equation*}
Weak convexity automatically guarantees that subgradients of $g$ satisfy the much stronger property \cite[Theorem 12.17]{RW98}:
\begin{equation}\label{eqn:stronger_ineq}
g(y)\geq g(x)+\langle v,y-x\rangle-\frac{\rho}{2}\|y-x\|^2,\qquad\forall x,y\in \R^d,~v\in \partial g(x).
\end{equation}
One important consequence we will use is the hypo-monotonicity inequality:
\begin{equation}\label{eqn:stronger_ineq_hypo}
\langle v - w,x - y\rangle \geq -\rho \|x-y\|^2,\qquad\forall x,y\in \R^d,~v\in \partial g(x),~w \in \partial g(y).
\end{equation}

The three assumption (A1), (A2), (A3) are standard in the literature on stochastic subgradient methods. Indeed, assumptions (A1) and (A2) are identical to assumptions (A1) and (A2) in~\cite{doi:10.1137/070704277}, while Assumption (A3) is the same as the assumption listed in~\cite[Equation~(2.5)]{doi:10.1137/070704277}.

In this work, we investigate the efficiency of the  proximal stochastic subgradient method, described in	Algorithm~\ref{alg:subgradient}.
\smallskip

	\begin{algorithm}[H]
		\KwData{$x_0 \in \dom(r)$, a sequence $\{\alpha_t\}_{t\geq 0}\subset\R_+$, and iteration count $T$}
		{\bf Step } $t=0,\ldots,T$:\\		
		\begin{equation*}\left\{
		\begin{aligned}
		&\textrm{Sample } \xi_t \sim dP\\
		& \textrm{Set } x_{t+1}=\prox_{\alpha_t r}\left(x_{t} - \alpha_t G(x_t, \xi_t)\right)
		\end{aligned}\right\},
		\end{equation*}
		Sample $t^*\in \{0,\ldots,T\}$ according to the probability distribution
		$\mathbb{P}(t^*=t)=\frac{\alpha_t}{\sum_{t=0}^T \alpha_t}.$
		{\bf Return} $x_{t^*}$		
		\caption{Proximal stochastic subgradient method
			}
		\label{alg:subgradient}
	\end{algorithm}
\smallskip

Henceforth, the symbol $\mathbb{E}_{t}[\cdot]$ will denote the expectation conditioned on all the realizations $\xi_0,\xi_1,\ldots, \xi_{t-1}$.

\subsection{Projected stochastic subgradient method}\label{sec:project_subgrad}
Our analysis of Algorithm~\ref{alg:subgradient} is shorter and more transparent when $r$ is the indicator function of a closed, convex set $\cX$. This is not surprising, since projected subgradient methods are typically much easier to analyze than their proximal extensions (e.g. \cite{prox_subgrad_duchi,cruz_subgrad}). Note that \eqref{eqn:problem_class} then reduces to the constrained problem
\begin{equation}\label{eqn:target_constr}
\min_{x\in\cX}~ g(x),
\end{equation}
and the proximal maps $\prox_{\alpha  r}(\cdot)$ become the nearest point projection $\proj_{\cX}(\cdot)$. Thus throughout Section~\ref{sec:project_subgrad}, we suppose that Assumptions (A1), (A2), and (A3) hold and that $r(\cdot)$ is the indicator function of a closed convex set $\cX$. The following is the main result of this section.

\begin{thm}[Stochastic projected subgradient method]\label{thm:stochastic_sub}
	Let $x_{t^*}$ be the point returned by Algorithm~\ref{alg:subgradient}. 
Then in terms of any constant  $\hat \rho>\rho$, the estimate holds:
\begin{align*}
 \EE \left[\|\nabla \varphi_{1/\widehat \rho}(x_{t^*})\|^2\right]
&\leq \frac{\hat{\rho}}{\hat \rho-\rho}\cdot\frac{(\varphi_{1/\hat \rho}(x_0) - \min \varphi )+ \frac{\hat\rho L^2}{2}\sum_{t=0}^T \alpha_t^2}{\sum_{t=0}^T \alpha_t}.
\end{align*} 
\end{thm}
\begin{proof}
	Let $x_{t}$ denote the points generates by Algorithm~\ref{alg:subgradient}. For each index $t$, define $\zeta_t := \EE_t[G(x_t, \xi)] \in \partial g(x_t)$ and set $\hat x_t:= \prox_{\varphi/\hat\rho}(x_t)$. We successively deduce
\begin{align}
\EE_t \left[\varphi_{1/\hat\rho}(x_{t+1})\right] &\leq \EE_t \left[g(\hat x_t) + \frac{\hat\rho}{2} \|x_{t+1} - \hat x_t\|^2\right] \label{eqn:prox_def} \\
&= g(\hat x_t) + \frac{\hat\rho}{2}\EE_t \left[\|\proj_{\cX}(x_{t} - \alpha_t G(x_t, \xi_t)) -  \proj_{\cX}(\hat x_t)\|^2\right] \notag\\
&\leq g(\hat x_t) + \frac{\hat \rho}{2}\EE_t \left[\|(x_{t}  - \hat x_t)- \alpha_t G(x_t, \xi_t)\|^2 \right]\label{eqn:nonexp_proof1}\\
&\leq g(\hat x_t) + \frac{\hat\rho}{2} \|x_{t} - \hat x_t\|^2 + \hat\rho\alpha_t\EE_t \left[\dotp{\hat x_t-x_t , G(x_t, \xi_t)}\right] + \frac{\alpha_t^2\hat\rho}{2}L^2\notag \\
&\leq \varphi_{1/\hat\rho}(x_{t})+ \hat\rho \alpha_t \dotp{\hat x_t - x_t , \zeta_t } + \frac{\alpha_t^2\hat\rho}{2}L^2 \notag\\
&\leq \varphi_{1/\hat\rho}(x_{t}) + \hat\rho \alpha_t \left(g(\hat x_t)-g(x_t)+\frac{\rho}{2}\|x_t-\hat x_t\|^2\right) + \frac{\alpha_t^2\hat\rho}{2}L^2, \label{eqn:weak_conv_proof}
\end{align}
where \eqref{eqn:prox_def} follows directly from the definition of the proximal map, the inequality \eqref{eqn:nonexp_proof1} uses that the projection $\proj_{\cX}(\cdot)$ is $1$-Lipschitz, and \eqref{eqn:weak_conv_proof} follows from weak convexity of $g$.


Using the law of total expectation to unfold this recursion yields: 
\begin{align*}
&\EE \left[\varphi_{1/\hat\rho}(x_{T+1})\right] \leq \varphi_{1/\hat\rho}(x_{0}) + \frac{\hat\rho L^2}{2}\sum_{t = 0}^T \alpha_t^2 - \hat\rho\EE\sum_{t = 0}^T\alpha_t \left(g(x_t) - g(\hat x_t) - \frac{\rho}{2}\|x_t - \hat x_t\|^2\right).
\end{align*} 
Lower-bounding the left-hand side by $\min \varphi$ and rearranging, we obtain the bound: 
\begin{equation}\label{eqn:main_ineq}
\begin{aligned}
 \frac{1}{\sum_{t=0}^T \alpha_t}\sum_{t=0}^T \alpha_t \EE\left[g(x_t) - g(\hat x_t) - \frac{\rho}{2}\|x_t - \hat x_t\|^2\right]
&\leq \frac{(\varphi_{1/\hat\rho}(x_{0}) - \min \varphi) +  \frac{\hat\rho L^2}{2}\sum_{t=0}^T \alpha_t^2}{\hat\rho\sum_{t=0}^T \alpha_t}.
\end{aligned}
\end{equation}
Notice that the left-hand-side of \eqref{eqn:main_ineq} is precisely  $\EE\left[g(x_{t^*}) - g(\hat x_{t^*}) - \frac{\rho}{2}\|x_{t^*} - \hat x_{t^*}\|^2\right]$.
Next, observe that the function $x\mapsto g(x)+\frac{\hat \rho}{2}\|x-x_{t^*}\|^2$ is  strongly convex with parameter $\hat\rho-\rho$, and therefore
\begin{align*}
&g(x_{t^*}) - g(\hat x_{t^*}) - \frac{\rho}{2}\|x_{t^*} - \hat x_{t^*}\|^2\\
&=
\left(g(x_{t^*})+\frac{\hat\rho}{2}\|x_{t^*} - x_{t^*}\|^2\right) - \left( g(\hat x_{t^*}) + \frac{\hat\rho}{2}\|x_{t^*} - \hat x_{t^*}\|^2\right) + \frac{\hat \rho - \rho}{2}\|x_{t^*} - \hat x_{t^*}\|^2\\
&\geq (\hat \rho-\rho)\|x_{t^*} - \hat x_{t^*}\|^2 = \frac{\hat \rho-\rho}{\hat\rho^2}\|\nabla \varphi_{1/\hat \rho}(x_{t^*})\|^2,
\end{align*}
 where the last equality follows from~\eqref{eqn:grad_form}. Using this estimate to lower bound the left-hand-side of \eqref{eqn:main_ineq} completes the proof.
\end{proof}

In particular, using the constant stepsize $\alpha$ on the order of $\frac{1}{\sqrt{T+1}}$ yields the following complexity guarantee. 

\begin{cor}[Complexity guarantee]\label{cor:complexity}
		
		Fix an index $T > 0$ and set the constant steplength $\alpha=\frac{\gamma}{\sqrt{T+1}}$ for some real $\gamma>0$. Then the point $x_{t^*}$ returned by Algorithm~\ref{alg:subgradient} satisfies:
		\begin{equation}\label{eqn:compl_proj}
		\EE\left[\|\nabla \varphi_{1/(2\rho) }(x_{t^*})\|^2\right]
		\leq 2\cdot\frac{\left(\varphi_{1/(2\rho)}(x_0) - \min \varphi\right)+ \rho L^2\gamma^2}{\gamma\sqrt{T+1}}.
		\end{equation} 
\end{cor}
\begin{proof}
This follows immediately from Theorem~\ref{thm:stochastic_sub} by setting $\hat \rho=2\rho$.
\end{proof}

Let us look closer at the guarantee of Corollary~\ref{cor:complexity} by minimizing out in $\gamma$. Namely, suppose we have available some real $ R>0$
satisfying $R\geq \varphi_{1/(2\rho)}(x_0) - \min \varphi$. We deduce from \eqref{eqn:compl_proj} the estimate,
$\EE\left[\|\nabla \varphi_{1/(2\rho) }(x_{t^*})\|^2\right]
\leq 2\cdot\frac{R+ \rho L^2\gamma^2}{\gamma\sqrt{T+1}}.$
 Minimizing the right-hand side in $\gamma$ yields the choice
$\gamma=\sqrt{\frac{R}{\rho L^2}}$ 
and therefore the guarantee
\begin{equation}\label{eqn:comp_bound_optim_proj}
\EE\left[\|\nabla \varphi_{1/(2\rho) }(x_{t^*})\|^2\right]\leq 4\cdot\sqrt{\frac{\rho R L^2}{T+1}}.
\end{equation}
In particular, suppose that $g$ is $L$-Lipschitz and the diameter of $\cX$ is bounded by some $D>0$. Then we may set $R:=  \min \left\{ \rho D^2, DL\right\}$, where the first term follows from the definition of the Moreau envelope and the second follows from Lipschitz continuity. 
Then the  number of subgradient evaluations required to find a point $x$ satisfying 
$\EE\|\nabla \varphi_{1/(2\rho) }(x)\|\leq \varepsilon$ is at most 
\begin{equation}\label{eqn:comple1}
\left\lceil 16\cdot\frac{(\rho L D)^2\cdot\min\left\{1,\tfrac{L}{\rho D}\right\}}{\varepsilon^4}\right\rceil.
\end{equation}
This complexity in $\varepsilon$ matches the guarantees of the stochastic gradient method for finding an $\varepsilon$-stationary point of a smooth function \cite[Corollary 2.2]{ghad}.


It is intriguing to ask if the complexity \eqref{eqn:comple1} can be improved when $g$ is a convex function. The answer, unsurprisingly, is yes. Since $g$ is convex, here and for the rest of the section, we will let the constant $\rho>0$ be arbitrary. As a first attempt, one may follow the observation of Nesterov \cite{nest_optima} for smooth minimization. The idea is that the right-hand-side of the complexity bound \eqref{eqn:comp_bound_optim_proj} dependence on the initial gap $\varphi(x_0)-\min \varphi$. We can make this quantity as small as we wish by a separate subgradient method.
Namely, we may simply run a subgradient method for $T$ iterations to decrease the gap $\varphi(x_0)-\min \varphi$ to $R:=LD/\sqrt{T+1}$; see for example \cite[Proposition 5.5]{nem_jud} for the this basic guarantee. Then we run another round of a subgradient method for $T$ iterations using the optimal choice $\gamma:=\sqrt{\frac{R}{\rho L^2}}$. A quick computation shows that the resulting scheme will find a point $x$ satisfying $\EE\|\nabla\varphi_{1/(2\rho)}(x)\|\leq \varepsilon$ after at most $O(1)\cdot\frac{L^2(\rho D)^{2/3}}{\varepsilon^{8/3}}$ iterations. 

This complexity can be improved slightly by first regularizing the problem. We will only outline the procedure here,  since the details are standard and easy to verify. Define the function $\widehat{\varphi}:=\varphi+\frac{\mu}{2}\|\cdot-x_{\rm c}\|^2$, for some $\mu>0$ and arbitrary $x_{\rm c}\in \cX$. We will apply optimization algorithms to $\widehat{\varphi}$ instead of $\varphi$, and therefore we must relate their Moreau envelopes.
 Fixing an arbitrary $\lambda>0$, it is straightforward to verify the following equality  by completing the square in the Moreau envelope:
$$\widehat{\varphi}_{1/\lambda}(x)=\varphi_{1/{(\lambda+\mu)}}\left(\tfrac{\mu}{\mu+\lambda}x_{\rm c}+\tfrac{\lambda}{\mu+\lambda}x\right)+\tfrac{\lambda\mu}{2(\mu+\lambda)}\|x-x_{\rm c}\|^2.$$ Differentiating in $x$ yields the bound
\begin{equation}\label{eqn:better_obtain_grad_est}
\left\|\nabla\varphi_{1/(\lambda+\mu)}\left(\tfrac{\mu}{\mu+\lambda} x_{\rm c}+\tfrac{\lambda}{\mu+\lambda}x\right)\right\|\leq \tfrac{\lambda+\mu}{\lambda}\|\nabla \widehat{\varphi}_{1/\lambda}(x)\|+\mu D.
\end{equation}
Thus, supposing $\varepsilon\leq 2\rho D$, we may set $\mu=\frac{\varepsilon}{2D}$ and $\lambda=2\rho-\frac{\varepsilon}{2D}$,
 obtaining the estimate
$$\left\|\nabla\varphi_{1/(2\rho)}\left(\tfrac{\mu}{\mu+\lambda} x_{\rm c}+\tfrac{\lambda}{\mu+\lambda}x\right)\right\|\leq 2\|\nabla \widehat{\varphi}_{1/\lambda}(x)\|+\tfrac{\varepsilon}{2}.$$
Hence, if we find a point $x$ satisfying $\EE\|\nabla \widehat{\varphi}_{1/\lambda}(x)\|\leq \tfrac{\varepsilon}{4}$, then the convex combination $z:=\tfrac{\mu}{\mu+\lambda} x_{\rm c}+\tfrac{\lambda}{\mu+\lambda}x$ would satisfy $\EE\|\nabla {\varphi}_{1/(2\rho)}(z)\|\leq \varepsilon$, as desired. Let us now apply the two-stage procedure on the strongly convex function $\widehat{\varphi}$. We first apply the projected stochastic subgradient method \cite{subgad_easier} for $T$ iterations yielding the estimate $R:=\frac{4(L^2 + \mu^2D^2) }{\mu(T+1)} = \frac{2D(4L^2 + \varepsilon^2) }{\varepsilon(T+1)}$. Then we apply the subgradient method (Algorithm~\ref{alg:subgradient}) for $T$ iterations on $\widehat{\varphi}$ with an optimal step-size $\gamma$. Solving for $T$ in terms of $\varepsilon$, a quick computation shows that the resulting scheme will find a point $z$ satisfying $\EE\|\nabla\varphi_{1/(2\rho)}(z)\|\leq \varepsilon$ after at most $O(1)\cdot \frac{(L^2+\varepsilon^2)\sqrt{\rho D}}{\varepsilon^{2.5}}$ iterations. By following a completely different technique, introduced by Allen-Zhu \cite{makegradsmall_zhu} for smooth stochastic minimization, this complexity can be even further improved to $\widetilde{O}\left(\frac{(L^2+\rho^2D^2)\log^3(\frac{\rho D}{\varepsilon})}{\varepsilon^2}\right)$ by running logarithmically many rounds of the subgradient method.  Since this procedure and its analysis is somewhat technical and is independent of the rest of the material, we have placed it in a supplementary text that can be found at \texttt{www.math.washington.edu/{\raise.17ex\hbox{$\scriptstyle\sim$}}ddrusv/sms.pdf}

\subsection{Proximal stochastic subgradient method}\label{sec:proxmeth}

We next move on to convergence guarantees of Algorithm~\ref{alg:subgradient} in full generality -- the main result of this work. To this end, in this section, in addition to assumptions (A1), (A2), and (A3) we will also  assume that $g$ is $L$-Lipschitz. 

We break up the analysis of Algorithm~\ref{alg:subgradient} into two lemmas. Henceforth, fix a real  $\hat \rho>\rho$. Let $x_t$  be the iterates produced by Algorithm~\ref{alg:subgradient} and let $\xi_t\sim dP$ be the i.i.d. realizations used. For each index $t$, define $\zeta_t := \EE_t[G(x_t, \xi)] \in \partial g(x_t)$ and set $\hat x_t:= \prox_{\varphi/\hat\rho}(x_t)$. Observe that by the optimality conditions of the proximal map, there exists a vector $\hat \zeta_t \in \partial g(\hat x_t)$ satisfying $\hat \rho(x_t - \hat x_t)  \in \partial r(\hat x_t) +\hat  \zeta_t$. The following lemma realizes $\hat x_t$ as a proximal point of $r$.
 
\begin{lem}\label{lem:prox_ident}
	For each index $t\geq 0$, equality holds:
	$$
	\hat x_t = \prox_{\alpha_t r}\left(\alpha_t \hat \rho x_t - \alpha_t\hat \zeta_t + (1-\alpha_t \hat \rho)\hat x_t\right).$$
\end{lem}
\begin{proof}
	By the definition of $\hat \zeta_t$, we have 
	\begin{align*}
	\alpha_t \hat \rho(x_t - \hat x_t)  \in \alpha_t\partial r(\hat x_t) + \alpha_t\hat \zeta_t
	&\iff \alpha_t \hat \rho x_t - \alpha_t\hat \zeta_t + (1-\alpha_t \hat \rho)\hat x_t  \in \hat x_t+\alpha_t\partial r(\hat x_t) \\
	&\iff\hat x_t = \prox_{\alpha_t r}(\alpha_t \hat \rho x_t - \alpha_t\hat \zeta_t + (1-\alpha_t \hat \rho)\hat x_t),
	\end{align*}
	where the last equivalence follows from the optimality conditions for the proximal subproblem. This completes the proof. 
\end{proof}

The next lemma establishes a crucial descent property for the iterates. 
\begin{lem}\label{lem:descent}
	Suppose $\hat\rho \in (\rho,2\rho]$ and we have $\alpha_t\leq {\hat \rho}^{-1}$ for all indices $t\geq 0$. Then the inequality holds:
	\begin{align*}
	\EE_t\|x_{t+1} - \hat x_t\|^2 
	&\leq  \|x_t - \hat x_t \|^2  + 2\alpha_t^2 L^2 - 2\alpha_t(\hat  \rho - \rho)\|x_t - \hat x_t\|^2.
	\end{align*}
\end{lem}
\begin{proof}
	We successively deduce
	\begin{align}
	&\EE_t\|x_{t+1} - \hat x_t\|^2\notag\\
	&=\EE_t\|\prox_{\alpha_t r} (x_t- \alpha_t G(x_t,\xi_t)) -  \prox_{\alpha_t r}(\alpha_t \hat \rho x_t - \alpha_t\hat \zeta_t + (1-\alpha_t \hat \rho)\hat x_t)\|^2\label{eqn:prox_real}\\
	&\leq \EE_t\|x_t - \alpha_t G(x_t, \xi_t) - (\alpha_t \hat \rho x_t - \alpha_t\hat \zeta_t + (1-\alpha_t \hat \rho)\hat x_t)\|^2\label{eqn:non_expproof_2} \\
	&= \EE_t\|(1-\alpha_t\hat  \rho)(x_t - \hat x_t) - \alpha_t( G(x_t, \xi_t) - \hat \zeta_t)\|^2\label{eqn:est_need_later}\\
	&= (1-\alpha_t\hat  \rho)^2\|x_t - \hat x_t \|^2 - 2(1-\alpha_t\hat  \rho)\alpha_t\EE_t\left[\dotp{x_t - \hat x_t, G(x_t, \xi_t) - \hat \zeta_t}\right] + \alpha_t^2 \EE_t\| G(x_t, \xi_t) - \hat \zeta_t\|^2\notag \\
	&= (1-\alpha_t\hat  \rho)^2\|x_t - \hat x_t \|^2 - 2(1-\alpha_t\hat  \rho)\alpha_t\dotp{x_t - \hat x_t, \zeta_t - \hat \zeta_t} + \alpha_t^2 \EE_t\| G(x_t, \xi_t) - \hat \zeta_t\|^2 \notag\\
	&\leq (1-\alpha_t\hat  \rho)^2\|x_t - \hat x_t \|^2 + 2(1-\alpha_t\hat  \rho)\alpha_t\rho \|x_t - \hat x_t\|^2 + 2\alpha_t^2 L^2\label{eqn:hypomon}  \\
	&= \|x_t - \hat x_t \|^2  + 2\alpha_t^2 L^2 - (2\alpha_t(\hat  \rho - \rho)  + \alpha_t^2\hat \rho( 2 \rho-\hat \rho))\|x_t - \hat x_t\|^2,\notag
	\end{align}
	where \eqref{eqn:prox_real} follows from Lemma~\ref{lem:prox_ident}, the  inequality \eqref{eqn:non_expproof_2} uses that $\prox_{\alpha_t r}(\cdot)$ is $1$-Lipschitz~\cite[Proposition 12.19]{RW98}, and 
	\eqref{eqn:hypomon} follows from the inequality~\eqref{eqn:stronger_ineq_hypo}. The result now follows from the assumed inequality $\hat \rho \leq 2\rho$. 
\end{proof}

With Lemma~\ref{lem:descent} proved, we can now establish convergence guarantees of Algorithm~\ref{alg:subgradient} in full generality. 

\begin{thm}[Stochastic proximal subgradient method]\label{thm:stochastic_sub2}
Fix a real $\hat \rho\in (\rho,2\rho]$ and a
stepsize sequence $\alpha_t\in (0,\hat \rho^{-1}]$.
Then the point $x_{t^*}$ returned by Algorithm~\ref{alg:subgradient} satisfies:
\begin{align*}
 \EE \left[\|\nabla \varphi_{1/\widehat \rho}(x_{t^*})\|^2\right]
&\leq \frac{\hat{\rho}}{\hat \rho-\rho}\cdot\frac{(\varphi_{1/\hat \rho}(x_0) - \min \varphi) + \hat\rho L^2\sum_{t=0}^T \alpha_t^2}{\sum_{t=0}^T \alpha_t}.
\end{align*} 
\end{thm}
\begin{proof}

We successively observe
\begin{align*}
\EE_t \left[\varphi_{1/\hat\rho}(x_{t+1})\right] &\leq \EE_t \left[\varphi(\hat x_t) + \frac{\hat\rho}{2} \|x_{t+1} - \hat x_t\|^2\right]  \\
&\leq \varphi(\hat x_t) + \frac{\hat \rho}{2}\left[\|x_t - \hat x_t \|^2+2\alpha_t^2L ^2 - 2\alpha_t(\hat  \rho - \rho)\|x_t - \hat x_t\|^2\right]\\
&= \varphi_{1/\hat{\rho}}( x_t)+\hat \rho\left[\alpha_t^2L ^2 - \alpha_t(\hat  \rho - \rho)\|x_t - \hat x_t\|^2\right],
\end{align*}
where the first inequality follows directly from the definition of the proximal map and the second follows from Lemma~\ref{lem:descent}.


Using the law of total expectation to unfold this recursion yields: 
\begin{align*}
&\EE \left[\varphi_{1/\hat\rho}(x_{T+1})\right] \leq \varphi_{1/\hat\rho}(x_{0}) + \hat\rho L^2\sum_{t = 0}^T \alpha_t^2 - \hat\rho(\hat \rho - \rho)\EE\sum_{t = 0}^T\alpha_t\|x_t - \hat x_t\|^2.
\end{align*} 
Next using the inequality $ \varphi_{1/\hat\rho}(x_{T+1})\geq \min \varphi$ and rearranging, we obtain the bound: 
\begin{equation}\label{eqn:main_ineq2}
\begin{aligned}
 \frac{1}{\sum_{t=0}^T \alpha_t}\sum_{t=0}^T \alpha_t \EE\left[\|x_{t} - \hat x_{t}\|^2\right]
&\leq \frac{(\varphi_{1/\hat \rho}(x_{0}) - \min \varphi) + \hat\rho L^2\sum_{t=0}^T \alpha_t^2}{\hat\rho(\hat \rho - \rho)\sum_{t=0}^T \alpha_t}.
\end{aligned}
\end{equation}
To complete the proof, observe that the left-hand-side is exactly $\EE\left[\|x_{t^*} - \hat x_{t^*}\|^2\right]$, while from from~\eqref{eqn:grad_form} we have the equality $\|x_{t^*} - \hat x_{t^*}\|^2 = (1/\hat\rho^2)\|\nabla \varphi_{1/\hat \rho}(x_{t^*})\|^2$.
\end{proof}

In particular, using the constant stepsize $\alpha$ on the order of $\frac{1}{\sqrt{T+1}}$ yields the following complexity guarantee.
\begin{cor}[Complexity guarantee]
		Fix a constant $\gamma\in (0, \tfrac{1}{2\rho}]$ and an index $T>0$, and set the constant steplength $\alpha=\frac{\gamma}{\sqrt{T+1}}$.
		Then the point $x_{t^*}$ returned by Algorithm~\ref{alg:subgradient} satisfies:
		\begin{align*}
		\EE\left[\|\nabla \varphi_{1/(2\rho) }(x_{t^*})\|^2\right]
		&\leq 2\cdot\frac{\left(\varphi_{1/(2 \rho)}(x_0) - \min \varphi\right)+ \rho L^2\gamma^2}{\gamma\sqrt{T+1}}.
		\end{align*} 
\end{cor}
\begin{proof}
This follows immediately from Theorem~\ref{thm:stochastic_sub2} by setting $\hat \rho=2\rho$.
\end{proof}

\subsection{Proximal stochastic gradient for smooth minimization}\label{sec:prox_gradient}
Let us now look at the consequences of our results in the setting when $g$ is $C^1$-smooth with $\rho$-Lipschitz gradient. Note, that then $g$ is automatically $\rho$-weakly convex. In this smooth setting, it is common to replace assumption (A3) with the finite variance condition:
\begin{itemize}
	\item[$\overline{(A3)}$]\label{it4} There is a real $\sigma \geq 0$ such that the inequality, $\EE_\xi\left[ \|G(x, \xi) - \nabla g(x)\|^2\right] \leq \sigma^2$, holds for all $x \in \dom r$. 
\end{itemize}

Henceforth, let us therefore assume that $g$ is $C^1$-smooth with $\rho$-Lipschitz gradient and Assumptions (A1), (A2), and $\overline{\rm (A3)}$ hold.

All of the results in Section~\ref{sec:proxmeth} can be easily modified to apply to this setting. In particular, Lemma~\ref{lem:prox_ident} holds verbatim, while Lemma~\ref{lem:descent} extends as follows.

\begin{lem}\label{lem:descent_smooth}
Fix a real $\hat \rho> \rho$ and a
 sequence $\alpha_t\in (0, \hat{\rho}^{-1}]$. Then the inequality holds:
	\begin{align*}
	\EE_t\|x_{t+1} - \hat x_t\|^2 
	&\leq  \|x_t - \hat x_t \|^2  + \alpha_t^2 \sigma^2- \alpha_t(\hat \rho - \rho)\|x_t - \hat x_t\|^2.
	\end{align*}
\end{lem}
\begin{proof}
By the same argument as in Lemma~\ref{lem:descent_smooth}, we arrive at the inequality \eqref{eqn:est_need_later} with ${\hat \zeta}_t=\nabla g({\hat x}_t)$. Adding and subtracting $\nabla g(x_t)$, we successively deduce
	\begin{align}
	&\EE_t\|x_{t+1} - \hat x_t\|^2\notag\\
	&= \EE_t\|(1-\alpha_t\hat  \rho)(x_t - \hat x_t) - \alpha_t( G(x_t, \xi_t) -  \nabla g({\hat x}_t))\|^2\notag\\
	&= \EE_t\|(1-\alpha_t\hat  \rho)(x_t - \hat x_t) - \alpha_t( \nabla g(x_t)- \nabla g(\hat x_t)) - \alpha_t( G(x_t, \xi_t) - \nabla g(x_t) )\|^2\notag\\	
	&= \|(1-\alpha_t\hat  \rho)(x_t - \hat x_t) - \alpha_t( \nabla g(x_t)- \nabla g(\hat x_t)) \|^2 + \alpha_t^2\EE_t\| G(x_t, \xi_t) - \nabla g(x_t) \|^2\label{eqn:cross_termcancel3}\\	
	&\leq (1-\alpha_t\hat  \rho)^2\|x_t - \hat x_t \|^2 - 2(1-\alpha_t\hat  \rho)\alpha_t\dotp{x_t - \hat x_t, \nabla g(x_t) - \nabla g(\hat x_t)} \notag\\
	&\hspace{20pt} + \alpha_t^2\| \nabla g(x_t)- \nabla g(\hat x_t) \|^2 + \alpha_t^2\sigma^2\label{eqn:expandsquare} \\
	&= (1-\alpha_t\hat  \rho)^2\|x_t - \hat x_t \|^2 + 2(1-\alpha_t\hat  \rho)\alpha_t\rho\|x_t - \hat x_t\|^2 + \rho^2\alpha_t^2 \|x_t - \hat x_t\|^2+\alpha_t^2 \sigma^2 \label{eqn:lip_hyp_stuff}\\
	&= \|x_t - \hat x_t \|^2  + \alpha_t^2 \sigma^2 - (2\alpha_t(\hat  \rho - \rho)  + \alpha_t^2\hat \rho( 2 \rho-\hat \rho) - \rho^2 \alpha_t^2)\|x_t - \hat x_t\|^2,\notag\\
	&= \|x_t - \hat x_t \|^2  + \alpha_t^2 \sigma^2 - \alpha_t(\hat \rho - \rho)(2- \alpha_t(\hat \rho - \rho))\|x_t - \hat x_t\|^2,\notag
	\end{align}
	where \eqref{eqn:cross_termcancel3} follows from assumption (A2), namely  $\EE_t G(x_t, \xi_t)=\nabla g(x_t)$, inequality \eqref{eqn:expandsquare} follows by expanding the square and using assumption $\overline{(A3)}$, and inequality \eqref{eqn:lip_hyp_stuff} follows from \eqref{eqn:stronger_ineq_hypo} and  Lipschitz continuity of $\nabla g$. The assumption $\hat \rho \geq \rho$ guarantees $2- \alpha_t(\hat \rho - \rho)\geq 1$. The result follows. 
\end{proof}

We can now state the convergence guarantees of the proximal stochastic gradient method. The proof is completely analogous to that of Theorem~\ref{thm:stochastic_sub2}, with Lemma~\ref{lem:descent_smooth} playing the role of Lemma~\ref{lem:descent}.

\begin{cor}[Stochastic proximal gradient method for smooth minimization]\label{cor:conv_guarant} \hfill \\
Fix a real $\hat \rho> \rho$ and a
stepsize sequence $\alpha_t\in (0, \hat{\rho}^{-1}]$.
Then the point $x_{t^*}$ returned by Algorithm~\ref{alg:subgradient} satisfies:
\begin{equation*}
 \EE \left[\|\nabla \varphi_{1/\widehat \rho}(x_{t^*})\|^2\right]
\leq \frac{2\hat{\rho}}{\hat \rho-\rho}\cdot\frac{(\varphi_{1/\hat \rho}(x_0) - \min \varphi) + \frac{\hat\rho \sigma^2}{2}\sum_{t=0}^T \alpha_t^2}{\sum_{t=0}^T \alpha_t}.
\end{equation*} 
In particular, setting $\alpha=\frac{\gamma}{\sqrt{T+1}}$ for some real $\gamma\in(0,\frac{1}{2\rho}]$ yields the guarantee
\begin{equation*}
 \EE \left[\|\nabla \varphi_{1/(2\rho)}(x_{t^*})\|^2\right]
\leq 4\cdot\frac{(\varphi_{1/\hat \rho}(x_0) - \min \varphi) + \rho \sigma^2\gamma^2}{\gamma\sqrt{T+1}}.
\end{equation*} 
\end{cor}

It is worth noting that in this smooth setting, the norm $\|\nabla \varphi_{\lambda}(x)\|$ is closely related to the magnitude of the {\em prox-gradient mapping} 
$$\mathcal{G}_{\lambda}(x)=\lambda^{-1}\left(x-\prox_{\lambda r}(x-\lambda\nabla g(x))\right).$$
This stationarity measure typically appears when analyzing proximal gradient methods (e.g. \cite{nest_conv_comp,Ghadimi2016mini}). It is straightforward to see that this measure is proportional to the norm of the gradient of the Moreau envelope, the quantity we have been using  \cite[Theorem 3.5]{prox_error}:
$$(1-\rho \lambda)\|\mathcal{G}_{\lambda}(x)\|\leq \|\nabla \varphi_{\lambda}(x)\|\leq (1+\rho \lambda)\|\mathcal{G}_{\lambda}(x)\|\qquad \textrm{ for all }x\in\R^d.$$
Hence, the convergence guarantees of Corollary~\ref{cor:conv_guarant} can be immediately translated in terms of $\|\mathcal{G}_{1/(2\rho)}(x_t^*)\|$, allowing for a direct comparison with previous results.